\documentclass[11pt]{amsart}

\usepackage{tikz}
\usepackage{graphicx}
\usepackage{amssymb}
\usepackage[shortlabels]{enumitem}
\usepackage{caption}
\usepackage{subcaption}
\usepackage{float}
\usepackage{hyperref}

\newtheorem{theorem}{Theorem}[section]

\newtheorem{lemma}[theorem]{Lemma}

\newtheorem{proposition}[theorem]{Proposition}

\newtheorem{conjecture}[theorem]{Conjecture}

\newtheorem{maintheorem}{Theorem}

\theoremstyle{definition}

\newtheorem{remark}[theorem]{Remark}

\theoremstyle{remark}

\numberwithin{equation}{section}
\numberwithin{figure}{section}

\newcommand{\Z}{\mathbb Z}
\newcommand{\R}{\mathbb R}
\newcommand{\N}{\mathbb{N}}
\newcommand{\T}{\mathbb{T}}
\newcommand{\eps}{\varepsilon}

\newcommand{\U}{\mathcal U}
\newcommand{\C}{\mathcal C}
\newcommand{\V}{\mathcal V}

\newcommand{\diff}{\operatorname{Diff}}
\newcommand{\phc}{\operatorname{Phc}}

\newcommand{\transv}{\pitchfork}
\newcommand{\zeroeq}{\stackrel{\circ}{=}}

\newcommand{\diag}{\mathop{\text{diag}}}

\DeclareMathOperator{\nuh}{Nuh}
\DeclareMathOperator{\pper}{Per}

\newcommand{\eval}[2][\right]{\relax
  \ifx#1\right\relax \left.\fi#2#1\rvert}

\begin{document}

\title{New examples of stably ergodic diffeomorphisms in dimension 3}
\author{Gabriel N\'u\~nez}
\address{Department of Exact and Natural Sciences, 
Faculty of Engineering and Technology,
Catholic University of Uruguay, 
Comandante Braga 2715, 
11.600 Montevideo, Uruguay.
}
\email{francisco.nunez@ucu.edu.uy}
\author{Davi Obata}
\address{Department of Mathematics, CNRS UMR 8088, Universit\'e de Cergy-Pontoise, 2, av.  Adolphe Chauvin F-95302 Cergy-Pontoise, France.}
\email{davi.obata@math.u-psud.fr}
\begin{thanks}
{DO  was supported by the projects ERC project 692925 NUHGD, and ANR BEKAM : ANR-15-CE40-0001.}
\end{thanks}
\author{Jana Rodriguez Hertz}
\address{1. Department of Mathematics,
Southern University of Science and Technology of China, No 1088, Xueyuan Rd., Xili, Nanshan District, Shenzhen, Guangdong 518055, China}
\address{2. SUSTech International Center for Mathematics}
\email{rhertz@sustech.edu.cn}
\begin{thanks}
{JRH was partially supported by NSFC 11871262 and NSFC 11871394}
\end{thanks}
 
\begin{abstract}
{We prove that in the isotopy class of any volume preserving partially hyperbolic diffeomorphism in a $3$-dimensional manifold, there is a non-partially hyperbolic stably ergodic diffeomorphism. In particular, we provide new examples of stably ergodic diffeomorphisms in 3-dimensional manifolds with respect to a smooth volume measure.}
\end{abstract}
\maketitle

\section{Introduction}

Let $M$ be a closed compact manifold and let $m$ be a smooth probability measure on $M$. A diffeomorphism $f:M \to M$ that preserves the measure $m$ is \emph{ergodic} if any measurable set $B$ such that $f(B) = B$ verifies $m(B) =0$ or $1$. Ergodicity means that from the point of view of the measure the system cannot be decomposed into smaller invariant pieces.

Denote by $\diff^r_m(M)$ the space of $C^r$-diffeomorphisms of $M$ that preserves the measure $m$. We say that a diffeomorphism $f\in \diff^1_m(M)$ is \emph{stably ergodic} if there exists  a $C^1$-neighborhood $\mathcal{U}$ of $f$ such that any diffeomorphism $g\in \mathcal{U} \cap \diff^2_m(M)$ is ergodic. 
(With this definition, $f$ could be stably ergodic but not ergodic. However, the generic stably ergodic diffeomorphism will be ergodic, since ergodicity is a $G_{\delta}$-property and $\diff^{2}_{m}(M)$ is dense in $\diff^{1}_{m}(M)$\cite{avila2010})  It is a natural problem to understand when such a property holds.

In 1969, Anosov proved in \cite{AN1969} that every uniformly hyperbolic (or Anosov), volume preserving $C^2$-diffeomorphism is ergodic. Since uniform hyperbolicity is a $C^1$-open property, it implies that uniformly hyperbolic systems are stably ergodic. This was the first example of a stably ergodic system. Anosov's proof is based on an argument introduced by Hopf in \cite{Hopf1939} to prove the ergodicity of the Liouville measure for the geodesic flow of compact surfaces with constant negative curvature. The key property needed to adapt Hopf argument for uniformly hyperbolic $C^2$-diffeomorphisms was the \emph{absolute continuity} of the stable and unstable foliations which was proved by Anosov and Sinai in \cite{AS1967}. It is not known if every Anosov, volume preserving $C^1$-diffeomorphism is ergodic. 

For many years this was the only known example of a stably ergodic system. It was only in 1994 that Grayson, Pugh and Shub proved in \cite{GPS1994} that the time one map of the geodesic flow of a surface of constant negative curvature is stably ergodic for the Liouville measure. The time one map of the flow mentioned above has a type of hyperbolicity called partial hyperbolicity.

A diffeomorphism $f:M\to M$ is \emph{partially hyperbolic} if there exists a $Df$-invariant decomposition of the tangent bundle $TM = E^{uu} \oplus E^c \oplus E^{ss}$ such that: the directions $E^{ss}$ and $E^{uu}$ are nontrivial; the direction $E^{ss}$ contracts uniformly and the direction $E^{uu}$ expands uniformly under the action of $Df$; the behavior of $E^c$ is dominated by the hyperbolic directions (see Section \ref{section.preliminaries}). 

% with the bundles $E^{ss}$ and $E^{uu}$ non-trivial, with the following property: there exists a riemannian metric that verifies for any $x\in M$ and unit vectors $v^* \in E^*_x$, for $*=ss, c, uu$, the inequalities
%\[
%\begin{array}{rlcll}
%\|Df(x)v^{ss}\| & < & \|Df(x)v^c\| & < & \|Df(x) v^{uu}\|\\
% \|Df(x) v^{ss}\| & < & 1 & <&  \|Df(x) v^{uu}\|.
%\end{array}
%\]
%If the system admits a splitting of the form $TM = E^{cs} \oplus E^{uu}$ or $TM = E^{ss} \oplus E^{cu}$ verifying the similar inequalities as above, then we say that it is \emph{weakly partially hyperbolic}.

Based on the proof of stable ergodicity in \cite{GPS1994}, Pugh and Shub conjectured in \cite{PS1995} that stable ergodicity is $C^r$-dense among volume preserving $C^r$-partially hyperbolic diffeomorphisms. Since then, this conjecture has motivated many works on this topic. Indeed, most works on stable ergodicity have been done in the partially hyperbolic scenario.

There are several cases in which Pugh-Shub's conjecture has been proved. In  \cite{HHU2008}, it was proved for $r= \infty$ among partially hyperbolic systems with one dimensional center; in \cite{HHTU11}, it was proved for $r=1$ among systems with two dimensional center; and in \cite{ACW2017} also for $r=1$ but among all partially hyperbolic systems. We remark that there are several other works regarding stable ergodicity for partially hyperbolic systems. 

The picture is much more incomplete outside the partially hyperbolic setting. In 2004, Tahzibi obtained in \cite{T2004} the first example of a stably ergodic system having no hyperbolic direction. There are some recent works by the authors in \cite{NRH2019, obata2019}, where they obtain results on stable ergodicity for non partially hyperbolic systems with some assumptions such as minimality of the strong unstable foliation, or chain hyperbolicity.

In 2012, the third author proposed the first author the following problem:
\begin{conjecture}\cite{NRH2019} Generically in $\diff^{1}_{m}(M)$, either all Lyapunov exponents are zero, or $f$ is stably ergodic and stably non-uniformly hyperbolic.  
\end{conjecture}
(The problem was stated for 3-manifolds, as Theorem \ref{thm.genericdichotomy} was only known to hold in this setting back then.)\par
We remark that all the known examples of stably ergodic systems that are not partially hyperbolic are isotopic to certain types of linear Anosov diffeomorphisms. These types of examples were first introduced by Bonatti and Viana in \cite{BV2000}, where they proved that they are robustly transitive, and later Tahzibi proved the stable ergodicity in \cite{T2004}. They are deformations of certain Anosov systems inside some fixed small balls in the manifold. In particular, all of such examples are on the torus.

In this work we give new examples of stably ergodic systems that are not partially hyperbolic on $3$-manifolds. Our main result is the following:

%(Some introduction talking about stable ergodicity of partially hyperbolic diffeomorphisms with respect to $m$, stable ergodicity of non-partially hyperbolic diffeomorphisms (Tahzibi's thesis, Bonatti-Viana) Here we provide new examples. This should take at least a page, I can start in a few days, Maybe we can write the introduction after the statement of the theorem). \newline\par
%\par

\begin{maintheorem}\label{main.isotopy} In the isotopy class of any volume preserving partially hyperbolic diffeomorphism of a $3$-dimensional manifold  
there is a stably ergodic diffeomorphism which is not partially hyperbolic.
\end{maintheorem}

We remark that if $f$ is a stably ergodic diffeomorphism of a $3$-manifold, then it preserves a dominated splitting of the two types of dominated splitting $TM = E^{uu} \oplus E^{cs}$ or $TM =  E^{cu} \oplus E^{ss}$, where the direcion $E^{ss}$ and $E^{uu}$ are uniformly contracting/expanding, respectively.

In order to prove Theorem \ref{main.isotopy}, we will prove the density of stable ergodicity inside a certain class of diffeomorphisms of $3$-dimensional manifolds. In what follows, we refer the reader to Section \ref{section.preliminaries} for the definitions of some basic dynamical objects that will appear.

 Let $\mathcal{V} \subset \diff^1_m(M)$ be the $C^1$-open set of diffeomorphisms $f$ satisfying:

\begin{enumerate}
 \item $f$ has a dominated splitting of the form $TM = E^{uu}_f \oplus E^{cs}_f$, where $E^{uu}$ is one-dimensional
 \item there exists a hyperbolic periodic point $p_{f}$ of stable index $2$ and an open set $U_{f}\ni p_{f}$ such that the local unstable manifold of each point in $U_{f}$ transversely intersect the local stable manifold of $p_{f}$. $U_{f}$ also satisfies:
\item the maximal invariant set $\Lambda_{f}=\bigcap_{n\in\Z}f^{n}(M\setminus U_{f})$ is partially hyperbolic. $\Lambda_{f}$ could be empty.
\end{enumerate}

\begin{maintheorem}\label{thm.maintheorem}
There exists a $C^1$-dense set $\mathcal{D}$ in $\mathcal{V}\cap \diff^2_m(M)$ such that any $f\in \mathcal{D}$ is stably ergodic. 
\end{maintheorem}

 In the setting of Theorem \ref{thm.maintheorem} there is only one hyperbolic direction and the usual strategy that is applied for partially hyperbolic systems does not apply.

This work is organized as follows. In Section \ref{section.preliminaries} we recall all the results we will use in the proof of Theorem \ref{thm.maintheorem}. Then in Section \ref{section.proofmainthm} we prove Theorem \ref{thm.maintheorem}. In Section \ref{section.thmA} we prove Theorem \ref{main.isotopy} by explaining how to deform any partially hyperbolic system to get in the setting of Theorem \ref{thm.maintheorem}. At last, in Section \ref{section.finalexamples} we make some final remarks and make explicit new classes of examples which were not known before.

\section{Preliminaries}\label{section.preliminaries}

We will say that a diffeomorphism $f\in\diff^{1}_{m}(M)$ admits a {\em dominated splitting} $TM=F\oplus E$ if both $E$ and $F$ are $Df$-invariant subbundles, and there exists a Riemannian metric such that for every $x\in M$ and $v_{F}\in F$, $v_{E}\in E$ unit vectors, we have
$$||Df(x)v_{E}||\leq \frac{1}{2}||Df(x)v_{F}||$$
\begin{remark}\label{remark.unidimensional}
It can easily be shown that if $\dim F=1$, then $F$ is a uniformly expanding bundle.  
\end{remark}

A point $p$ is periodic if there exists $l \in \N$ such that $f^l(p) = p$. The smallest natural number that verifies $f^l(p)=p$ is called the period of $p$ and we denote it by ${\rm per}(p)$. A periodic point is \emph{hyperbolic} if all the eigenvalues of $Df^{{\rm per}(p)}(p): T_pM \to T_pM$ have absolute value different than $1$. The stable manifold of $p$ is defined by
\[
W^s(p,f):= \{ x\in M: \lim_{n\to +\infty} d(f^n(x), f^n(p) ) = 0\}.
\]
By the Stable manifold Theorem, if $p$ is a hyperbolic periodic point then $W^s(p,f)$ is a $C^1$-immersed submanifold. One defines the unstable manifold of $p$ in a similar way but considering past iterates of $f$, and denote it by $W^u(p,f)$. We write $\pper_H(f)$ as the set of hyperbolic periodic points of $f$.

We say that a number $\lambda \in \R$ is a Lyapunov exponent for the point $x \in M$ and vector $v\in T_xM$ if  
\[
\lambda (x,v):= \lim_{n\to +\infty} \frac{1}{n} \log \|Df^n(x) v\| = \lambda.
\]

Oseldets' theorem states that for $m$-almost every point $x$, there exist $k(x)$ different Lyapunov exponents $\lambda_1(x)> \cdots > \lambda_{k(x)}(x)$, where $1\leq k(x) \leq \mathrm{dim}(M)$, and a $Df$-invariant decomposition (called the \emph{Oseledets decomposition}) 
\[
T_xM = E^1_x \oplus \cdots \oplus E^{k(x)}_x
\]
such that $\lambda_i(x) = \lambda(x,v_i)$, for any non-zero vector $v_i \in E^i_x$ with $i=1, \cdots, k(x)$. For $m$-almost every point $x$ we consider the decomposition
\[
T_xM = E^+_x \oplus E^0_x \oplus E^-_x,
\]
such that $\lambda(x,v_-)<0$ for $v_-\in E^-_x/\{0\}$; $\lambda(x,v_0) = 0$ for $v_0\in E^0_x/\{0\}$; and $\lambda(x,v_+) >0 $ for  $v_+ \in E^+_x/\{0\}$. 

Consider the sets 
\[\arraycolsep=1.2pt\def\arraystretch{1.2}
\begin{array}{rcl}
W^-(x,f) & := & \{ y\in M: \limsup_{n\to +\infty} \frac{1}{n} \log d(f^n(x), f^n(y)) <0 \},\\
W^+(x,f) & := & \{ y\in M: \limsup_{n\to + \infty} \frac{1}{n} \log d(f^{-n}(x), f^{-n}(y) <0\}.
\end{array}
\]
In \cite{Pesin1977}, Pesin proved that if $f$ is $C^2$ then for $m$-almost every $x$ the sets $W^-(x)$ and $W^+(x)$ are $C^1$-immersed submanifolds of $M$.

Following \cite{HHTU11}, given a hyperbolic periodic point $p\in M$, we define its {\em stable Pesin homoclinic class} by 
\begin{equation}\label{eq.phc.-}
\phc^{-}(p)=\{x: W^{-}(x,f)\transv W^{u}(o(p),f)\ne\emptyset\} 
\end{equation}
where $W^{u}(o(p))$ is the union of the unstable manifolds of $f^{k}(p)$, for all $k=0,\dots,{\rm per}(p)-1$. $\phc^{-}(p)$ is invariant and saturated by $W^{-}$-leaves. Analogously, we define 
\begin{equation}\label{eq.phc.+}
\phc^{+}(p)=\{x: W^{+}(x,f)\transv W^{s}(o(p),f)\ne\emptyset\} 
\end{equation}
Let us denote by $\nuh(f)$ the set of points $x\in M$ such that all Lyapunov exponents in $x$ are different from zero.

\begin{theorem}[Theorem A, \cite{HHTU11}]\label{criterion_hhtu} Let $f:M \rightarrow M$ be a $C^2$-diffeomorphism over a closed connected Riemannian manifold $M$, let $m$ be a smooth invariant measure and $p \in \pper_H(f)$. If $m(\phc^{+}(p))>0$ and $m(\phc^{-}(p))>0$, then
\begin{enumerate}
	\item $\phc^{+}(p) \circeq \phc^{-}(p) \circeq \phc(p)$.
	\item $m|\phc(p)$ is ergodic.
	\item $\phc(p) \subset \nuh(f)$.
\end{enumerate}
\end{theorem}
The notation $A\zeroeq B$ means $m(A\triangle B)=0$, where $A\triangle B$ is the symmetric difference between the sets $A$ and $B$. Recall that a residual subset of $\diff^1_m(M)$ is a set that contains a $C^1$-dense $G_{\delta}$ subset of $\diff^1_m(M)$. 

\begin{theorem}\cite{AB2012}\label{AB.periodic} There exists a residual subset $\mathcal{R}$ of $\diff^{1}_{m}(M)$ such that for $f\in \mathcal{R}$, if $m(\nuh(f))>0$ then there exists a hyperbolic periodic point $q$ such that
$$\phc(q)\circeq \nuh(f)$$ 
and $m|_{\nuh(f)}$ is $f$-ergodic.
\end{theorem}

\begin{theorem} \cite{JRH2012}\label{thm.genericdichotomy}  Let $M$ be a closed connected manifold of dimension 3, then there exists $\mathcal{R} \subset \diff^{1}_m (M)$ a residual set such that every $f \in \mathcal{R}$ satisfies one of the following alternatives:
\begin{itemize}
    \item All Lyapunov exponents of $f$ vanish almost everywhere, or
    \item $\nuh(f)\circeq M$
\end{itemize}
\end{theorem}

In \cite{ACW16}, the authors generalize Theorem \ref{thm.genericdichotomy} to arbitrary dimensional manifolds. 

\begin{theorem}\cite{BV2005}\label{bv05} There exists a residual set $\mathcal{R} \subset \diff^{1}_m (M)$ such that, for each $f\in \mathcal{R}$ and $m$-almost every $x\in M$, the Oseledets' splitting of $f$ is either trivial or dominated at $x$
\end{theorem}

For simplicity, let us sum all these statements in one:
\begin{lemma}\label{lemma.dichotomy} There exists  $\mathcal{R} \subset \diff^{1}_{m}(M)$ a residual subset such that if $f\in \mathcal{R}$, either:
\begin{itemize}
 \item all Lyapunov exponents of $f$ vanish almost everywhere, or
 \item $f$ is ergodic, the Oseledets splitting is globally dominated, and 
 there exists a hyperbolic periodic point $q_{f}$ such that 
 $$\phc(q_{f})\circeq\nuh(f) \circeq M$$
\end{itemize}
\end{lemma}

The \emph{unstable index} of a periodic point $p$ is the number of eigenvalues of $Df^{{\rm per}(p)}(p)$ with absolute value larger than $1$. We say that two hyperbolic periodic points $p,q\in \pper_H(f)$ are \emph{homoclinically related} if
\[
W^s(o(p),f)\pitchfork W^u(o(q),f) \neq \emptyset \textrm{ and } W^u(o(p),f) \pitchfork W^s(o(q),f) \neq \emptyset.
\]
We will use the following result:
\begin{theorem} \cite{AC2012}\label{thm.genericpointsrelated} There exists $\mathcal{R} \subset \diff^{1}_m(M)$ a residual subset such that if $f\in \mathcal{R}$, then any two periodic points $p,q\in \pper_H(f)$ with the same index are homoclinically related.
\end{theorem}
Let $PH^r_m(M)$ be the set of $C^r$-partially hyperbolic diffeomorphisms of $M$ that preserve $m$. We will also need the following result.

\begin{theorem}\cite{ACW2017}\label{ACW2017} Stable ergodicity is $C^1$-dense in $PH^r_m(M)$ for any $r>1$.
\end{theorem}

\section{Proof of Theorem \ref{thm.maintheorem}} \label{section.proofmainthm}

Let $f$ be a generic diffeomorphism in $\mathcal{V}$, and let $p_f$ be the fixed point from the definition of $\mathcal{V}$. Our strategy will be to show that either $f$ is partially hyperbolic (hence $f$ is $C^1$-approximated by stably ergodic diffeomorphisms, after Theorem \ref{ACW2017}), or there exists a neighborhood ${\mathcal U}$ of $f$ such that for all $g\in \U\cap\diff^{2}_{m}(M)$, the analytic continuation $p_{g}$ of $p_{f}$ satisfies
$$\phc_{g}(p_{g})\circeq M$$
whence $f$ is also stably ergodic. This will prove Theorem \ref{thm.maintheorem}. \newline\par
Since $\dim E^{uu}=1$, $E^{uu}$ is uniformly expanding, hence, almost every point has at least one positive Lyapunov exponent. Lemma \ref{lemma.dichotomy} then implies that there exists a hyperbolic periodic point $q_{f}$ such that $$\phc(q_{f})\circeq M.$$
If the unstable index of $q_{f}$ is $2$, then, since the Oseledets splitting is dominated we would have a dominated splitting of the form $TM = E^{cu} \oplus E^{ss}$, where $E^{ss}$ is one-dimensional. By Remark \ref{remark.unidimensional}, the direction $E^{ss}$ is uniformly contracting. Since we already had that $TM=E^{uu}\oplus E^{cs}$, with one-dimensional $E^{uu}$ in $\mathcal{V}$, then we must have that there is a dominated  splitting of the form $TM=E^{uu}\oplus E^{c}\oplus E^{ss}$ with one-dimensional subbundles. Hence $f$ is partially hyperbolic, and therefore, by Theorem \ref{ACW2017}, $f$ is $C^1$-approximated by stably ergodic diffeomorphisms. \par
We may therefore assume that the unstable index of $q_{f}$ is one, and therefore, by Theorem \ref{thm.genericpointsrelated} and the following lemma, we have that $\phc(q_{f})=\phc(p_{f})$:
\begin{lemma}\cite{HHTU11} If $p$ and $q$ are homoclinically related, then 
 $$\phc(p)=\phc(q)$$
\end{lemma}

Therefore, we have that $\phc(p_{f})\circeq M$. Since $f\in\V$ is $C^1$-generic, we have the following:

\begin{lemma} \cite[Lemma 5.1]{AB2012} \label{lem.robustnessphminus} For every $\eps>0$ there exists a $C^{1}$-neighborhood $\U(f)$ of $f$ such that for all $C^{2}$ diffeomorphisms $g$ in $\U(f)$:
 $$m(\phc_{g}(p_{g}))>1-\eps$$
 where the hyperbolic periodic point $p_{g}$ is the analytic continuation of $p_{f}$. 
\end{lemma}

The proof of Theorem \ref{thm.maintheorem} will be ended if we prove the following:
\begin{proposition}\label{prop.phcueentire}
Any $C^2$-diffeomorphism $g$ which is sufficiently $C^1$-close to $f$ verifies $\phc^+(p_g) \circeq M.$
\end{proposition}

It will be convenient to introduce the following notation:
$$
\Gamma(g)=M \setminus \phc^+(p_g).
$$ 
Let us call $\Gamma(g)$ the {\em bad set}.
In the following lemma we state some basic properties of $\Gamma(g)$. 

\begin{lemma} The set $\Gamma(g)$ is compact, $g$-invariant, $u$-saturated. $C^1$-generically it does not contain periodic points of unstable index $1$.
\end{lemma}
\begin{proof}[Proof of the lemma] Clearly, $\Gamma(g)$ is $g$-invariant and $u$-saturated. By the Stable Manifold Theorem we have if $W^{uu}(x,g) \pitchfork W^s(o(p_g),g) \neq \emptyset$ then there exists an open neighborhood $V$ of $x$ in $M$ such that for every $y \in V$ we have $W^{uu}(y,g) \pitchfork W^s(o(p),g) \neq \emptyset$. Hence, $\phc^+(p_g)$ is an open set. Therefore, $\Gamma(g)$ is a compact set. By Theorem \ref{thm.genericpointsrelated}, $C^1$-generically in $\diff^{1}_m(M)$, all periodic points of the same index are homoclinically related, then $C^1$-generically $\Gamma(g)$ has no periodic points with unstable index equal to the unstable index of $p_g$. This completes the proof of the lemma.
\end{proof}

\begin{lemma}[\cite{NRH2019}, Lemma $3.2$]\label{lem.haussemicontinuity}
The function $h\mapsto \Gamma(h)$ is upper semicontinuous, that is, if $(h_n)_{n\in \mathbb{N}}$ is a sequence converging to $h$ in the $C^1$-topology then $\limsup \Gamma(h_n) \subset \Gamma(h)$. Equivalently, for every $h\in\V$ and every $\eps>0$ there exists a $C^{1}$-neighborhood $\U\subset\diff^{1}_{m}(M)$ such that if $g\in\U$, then $\Gamma(g)\subset B_{\eps}(\Gamma(h))$, where $B_{\eps}(A)=\{x\in M: d(x,A)<\eps\}$.
\end{lemma}

\begin{lemma}\label{lemma.PH}
There exists a neighborhood $\U$ of $f$ such that for all $g\in\U$, the set $\Gamma(g)$ is partially hyperbolic. That is, it admits a decomposition $TM=E^{uu}\oplus E^{c}\oplus E^{ss}$ .
\end{lemma}

\begin{proof} From the definition of $\V$, it follows that the open set $U_{f}$ is contained in $\phc^{+}(p_{f})$. Therefore, $\Gamma(f)\subset \Lambda_{f}$ the maximal invariant set of $M\setminus U_{f}$, which is partially hyperbolic. By continuity of partial hyperbolicity, there exists $\eps>0$ and $\U_{1}$ such that every $g$-invariant set contained in $B_{\eps}(\Gamma(f))$, with $g\in \U_{1}$ is partially hyperbolic. Since $g\mapsto \Gamma(g)$ is upper semicontinuous, there exists a $C^{1}$-neighborhood $\U\subset \U_{1}$ such that $\Gamma(g)\subset B_{\eps}(\Gamma(f))$. This proves the claim. 
\end{proof}

\begin{proof}[Proof of Proposition \ref{prop.phcueentire}]
Suppose that the conclusion of Proposition \ref{prop.phcueentire} does not hold. Then, there exists a sequence of $C^2$-diffeomorphisms $(f_n)_{n\in \mathbb{N}}$ converging to $f$ in the $C^1$-topology such that $m(\Gamma(f_n))>0$. For each $n\in \mathbb{N}$, the set $\Gamma(f_n)$ satisfies the following property: for any $ x\in \Gamma(f_n)$ the set $W^{ss}(x,f_n) \cup W^{uu}(x,f_n)$ is contained in $\Gamma(f_n)$. This follows from Lemma $6.1$ in \cite{JRH2012}. In particular, the set $\Gamma(f_n)$ is $su$-saturated for any $n\in \mathbb{N}$.

Let $\Gamma' = \limsup \Gamma(f_n)$, and observe that $\Gamma'$ is a non-empty compact set. By Lemma \ref{lem.haussemicontinuity}, we have that $\Gamma' \subset \Gamma(f)$. Since the set $\Gamma(f_n)$ is $f_n$-invariant, for each $n\in \mathbb{N}$, it follows that $\Gamma'$ is $f$-invariant. Moreover, since each set $\Gamma(f_n)$ is $su$-saturated, we conclude that $\Gamma'$ is also $su$-saturated. Indeed, every point $x$ in $\Gamma'$ is the limit of a converging sequence of points $x_{n}\in \Gamma(f_{n})$. The $f_{n}$-stable and unstable manifolds of $x_{n}$ are contained in $\Gamma(f_{n})$. 
Since these stable and unstable manifolds are hyperbolic ($\Gamma(f_{n})$ is partially hyperbolic), they vary continuously with respect to the variable $x$ and to the variable $f$, so their limits are the $f$-stable and unstable manifolds of $x$. Therefore the $f$-stable and unstable manifolds of $x$ are contained in $\Gamma'$. \par

From the results in chapter $6$ of \cite{JRH2012}, there are two periodic points $q_1, q_2 \in \Gamma'$ such that $W^{ss}(q_1,f) \cap W^{uu}(q_2,f) \neq \emptyset$. However, this does not happen generically. Since $f$ is a $C^1$-generic diffeomorphism, we get a contradiction. Hence, for any $C^2$-diffeomorphism $g$ in a $C^1$-neighborhood of $f$ we have that $m(\Gamma(g)) = 0$.
\end{proof}

Now we are ready to prove Theorem \ref{thm.maintheorem}.

\begin{proof}[Proof of Theorem \ref{thm.maintheorem}]
Let $f$ be a $C^1$-generic diffeomorphism in $\mathcal{V}$ and let $g$ be a volume preserving $C^2$-diffeomorphism which is $C^1$-close to $f$. By Lemma \ref{lem.robustnessphminus}, we have that $m(\phc^-(p_g)) >0$. By Proposition \ref{prop.phcueentire}, we obtain that $m(\phc^u(p_g))=1$. By the criterion of ergodicity (Theorem \ref{criterion_hhtu}) we conclude that $g$ is ergodic. 
\end{proof}

\section{Proof of Theorem \ref{main.isotopy}} \label{section.thmA}
Let $f_{0}:M^{3}\to M^{3}$ be a volume preserving partially hyperbolic diffeomorphism in a closed Riemannian $3$-manifold. Assume there exists a periodic point $p$ such that $f_0$ is $C^{1}$-linearizable in a neighborhood of $p$. Call $F_{0}:\R^{3}\to\R^{3}$ the linearization of $f_{0}$ in a suitable neighborhood $U$ of $p$. Assume there exist $\mu,\rho,\lambda\in \R$ so that $F_{0}$ can be represented by a matrix $\diag(\mu,\rho,\lambda)$. 
We will start by asking that:
\begin{enumerate}
 \item $p$ be fixed
 \item $0<\lambda<\rho<1<\mu$
\end{enumerate}
We will see later that we can always reduce the case to this situation. We will consider a volume preserving $C^{0}$-perturbation $F_{1}$ of $F_{0}$ supported in $U\subset \bigcup_{x\in W^s_{loc}(p)} W^{uu}_{\eps}(x)$ so that 
\begin{enumerate}
\item $DF_{1}(0,0,0)=\diag(\mu, \mu^{-\frac12},\mu^{-\frac12})$ and $DF_{1}(x,y,z)=F_{0}$ outside a small ball $B_{\eps}(0)$ 
\item a cone field of the form $$\mathcal{C}^{uu}_{\gamma}=\{v:\|(\pi_{2}+\pi_{3})v\|\leq \gamma\|\pi_{1}v\|\}$$ satisfies 
\begin{equation}\label{eq.cones}
DF(x)\C^{uu}_{\gamma}\subset\C^{uu}_{\xi} 
\end{equation}
for some $0<\xi<\gamma$
\end{enumerate}

It is easy to see that we can also choose $\gamma',\xi'$, so that $$DF^{-1}(x)\C^{cs}_{\gamma'}\subset\C^{cs}_{\xi'},$$ where $C^{cs}_{\gamma}=\overline{\R^{3}\setminus\C^{uu}_{\gamma}}$. We will follow the construction of  \cite{Ma1978} and \cite{BV2000}, but in a volume preserving setting. 
For completeness, we include a short explanation of how this is done. 
\subsection*{Step 1}
Let $\eta=\frac{\sqrt{\mu}}{\lambda}>1$. We will find a flow $\phi_{t}$ of a divergence-free vector field $X$, so that $D\phi_{1}(0,0)=\diag(\eta^{-1},\eta)$.\par
Consider a smooth function $\psi:\R\to\R$ so that $\psi'(0)=\sqrt{\log\eta}$ and $\psi(y)=0$ outside a small ball $B_{\eps}(0)$. We can choose $\psi$ so that $\psi(y)\psi''(y)$ is arbitrarily small. The field $X(y,z)=(\psi(y)\psi'(z),-\psi'(y)\psi(z))$ is divergence-free. Let $\phi_{t}$ be the flow of $X$, that is $\frac{\partial}{\partial t}\phi_{t}(y,z)=X(\phi_{t}(y,z))$. Then $D\phi_{t}(x,y)$ satisfies the initial value problem $\frac{\partial}{\partial t}D\phi_{t}(y,z)=DX(\phi_{t}(y,z))D\phi_{t}(y,z)$ with $D\phi_{0}(y,z)=id$. In particular $D\phi_{t}(0,0)=\exp (DX(0,0)t)=\diag(\eta^{-t},\eta^{t})$, and therefore $DH(0,0)=D\phi_{1}(0,0)=\diag(\eta^{-1},\eta)$. Since $X(y,z)=(0,0)$ outside a small ball $B_{\eps}(0,0)$, it is easy to see that $\phi_{t}(y,z)=(y,z)$ outside $B_{\eps}(0,0)$.

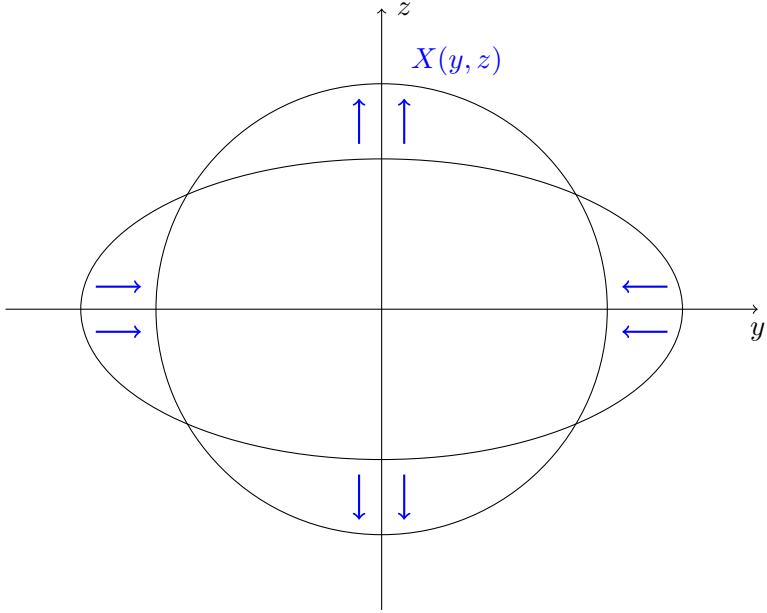
\begin{figure}
 \begin{tikzpicture}
 \draw[->] (-5,0)--(5,0);
\draw (5,-0.3) node {$y$};
\draw (0.3,4) node {$z$};
 \draw[->] (0,-4)--(0,4);
 \draw (0,0) ellipse (4cm and 2cm);
 \draw (0,0) circle (3cm);
 \draw[->, thick, color=blue] (3.8,0.3)--(3.2,0.3);
  \draw[->, thick, color=blue] (3.8,-0.3)--(3.2,-0.3);
   \draw[->, thick, color=blue] (-3.8,0.3)--(-3.2,0.3);
  \draw[->, thick, color=blue] (-3.8,-0.3)--(-3.2,-0.3);
 \draw[<-, thick, color=blue] (0.3,2.8)--(0.3,2.2); 
  \draw[<-, thick, color=blue] (-0.3,2.8)--(-0.3,2.2); 
 \draw[<-, thick, color=blue] (0.3,-2.8)--(0.3,-2.2); 
  \draw[<-, thick, color=blue] (-0.3,-2.8)--(-0.3,-2.2); 
  \draw[thick, color=blue] (1,3.3) node {$X(y,z)$};
\end{tikzpicture} 
\caption{Step 1: the divergence-free vector field $X(y,z)$}
\end{figure}

\subsection*{Step 2}  
Now, let $g:\R^{3}\to\R^{2}$ be $g(x,y,z)=\phi_{t(x)}(\rho y,\lambda z)$, where $t:\R\to\R$ is a smooth function such that $t(0)=1$ and $t(x)=0$ outside a small $B_{\eps}(0)$. Then, 
the Jacobian of $g$ with respect to $yz$ satisfies:

$$J_{yz}(g)(x,y,z)=\left|
\begin{array}{cc}
 g_{1y}&g_{1z}\\
 g_{2y}&g_{2z}
\end{array}
\right|=\mu^{-1}\qquad\text{in}\quad\R^{3}$$

and the derivatives satisfy $|g_{*}|\leq \rho$ for all $*=1y,1z,2y,2z$ and $|g_{*}|\leq K$ for $*=1x,2x$ for some $K>0$. Let $a>0$ be such that $aK<\eps$. If we consider the map $F_{1}(x,y,z)=(\mu x,ag(x,a^{-1}y,a^{-1}z))$, then $F_{1}$ is isotopic to $F_{0}$ (in fact, $C^{0}$ close)
$F_{1}$ preserves volume since, as it is easy to check, $J_{yz}(g_{a})=\mu^{-1}$, where $g_{a}(x,y,z)=ag(x,a^{-1}y,a^{-1}z)$. Moreover, 
$$DF_{1}(x,y,z)=\left(
\begin{array}{ccc}
\mu&0&0\\
ag_{1x}&g_{1y}&g_{1z}\\
ag_{2x}&g_{2y}&g_{2z} 
\end{array}
\right)$$

$g_{*}$ is evaluated in $(x,a^{-1}y,a^{-1}z)$ for all $*=1x,1y,1z,2x,2y,2z$. 

\begin{figure}[h]
\begin{subfigure}{0.45\textwidth}
 \includegraphics[width=\textwidth]{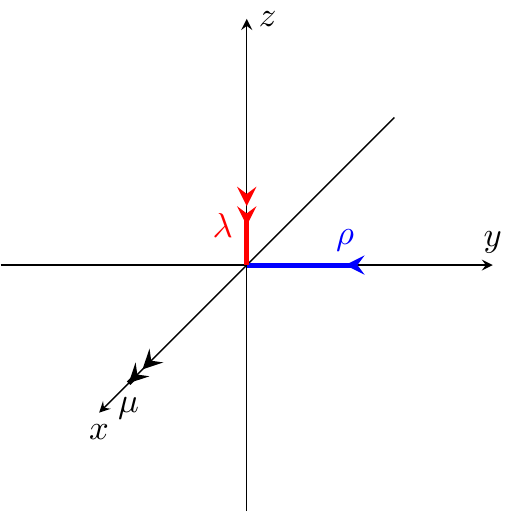}
\end{subfigure} \quad
\begin{subfigure}{0.45\textwidth}
 \includegraphics[width=\textwidth]{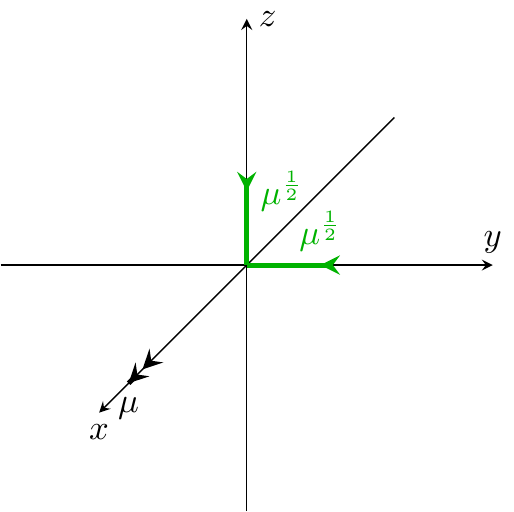}
\end{subfigure} 
\caption{Step 2: changing the contracting eigenvalues}
\end{figure}

\subsection*{Step 3}

Let us see that equation $(\ref{eq.cones})$ holds. Let $v'=DF_1(x)v$, with $v\in\C^{uu}_{\gamma}$, then 
$$\|(\pi_{2}+\pi_{3})v'\|\leq \eps|v_{1}|+\rho\|(\pi_{2}+\pi_{3})v\|\leq (\eps+\rho\gamma)|v_{1}|<\eta\mu |v_{1}|$$
if we choose $\frac{\eps+\rho\gamma}{\mu}<\eta<\gamma$.
\subsection*{Step 4}
Finally, we can compose $F_{1}$ with a small rotation $\alpha$ in the plane $yz$, supported in $U$, obtaining a volume preserving $F$ so that $DF(0)$ is a fixed point with an eigenvalue $\mu$ and two complex eigenvalues of modulus $\frac{1}{\sqrt{\mu}}$. Since the derivative of $\alpha$ can be made arbitrarily close to the identity, we can see that $DF(x)$ preserves cone fields, which we continue to call $\C^{uu}$ and $\C^{cs}$. This implies there is a global dominated splitting $TM=E^{uu}\oplus E^{cs}$. Since $E^{uu}$ is one-dimensional, it is hyperbolic.\newline\par
\begin{figure}[h]
\begin{subfigure}{0.45\textwidth}
 \includegraphics[width=\textwidth]{desenho3.pdf}
\end{subfigure} \quad
\begin{subfigure}{0.45\textwidth}
 \includegraphics[width=\textwidth]{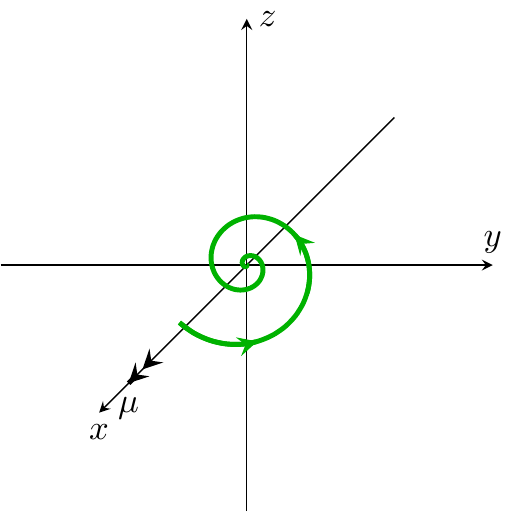}
\end{subfigure} 
\caption{Step 4: composing with a small rotation supported in $U$}
\end{figure}

 Call $f$ the diffeomorphism on $M$ obtained in this way, that is, so that its linearization in a neighborhood of the fixed point $p$ is $F$. Because of the form of our construction, 

\begin{equation}\label{U.contained}
U_{f}\subset \displaystyle \bigcup_{x\in W^s_{loc}(p)} W^{uu}_{\varepsilon}(x,f); 
\end{equation}

Also by construction, $f$ is partially hyperbolic on $M\setminus U$, therefore $\Lambda_{f}$, the maximal invariant of $M\setminus U_{f}$, is a partially hyperbolic set. By Theorem \ref{thm.maintheorem}, there exist stably ergodic diffeomorphisms arbitrarily $C^1$-close to $f$ which are not partially hyperbolic. Observe that by our construction such diffeomorphisms are isotopic to $f_0$.

\subsection*{Step 5}
If $p$ were a periodic point of period $n$ instead of a fixed point, we could proceed as follows. By applying the same construction for $f_0^n$ instead of $f_0$, we obtain a diffeomorphism $f_n$ which coincides with $f_0^n$ outside a small ball $B_{\varepsilon}(p)$ and that verifies the same properties that we obtained for $f$.

Consider $h = f_n \circ f^{-n}_0$. This diffeomorphism is the identity outside $B_{\varepsilon}(p)$. We may suppose that $\varepsilon$ is small enough such that for each $i, j \in \{0, \cdots, n-1\}$ with $i\neq j$ then $f_0^i (B_{\varepsilon}(p)) \cap f^j_0(B_{\varepsilon}(p)) = \emptyset$.

Let $f = h \circ f_0 $. Observe that for any $q\in B_{\varepsilon}(p)$ and for $j=0 ,\cdots, n-2$ we have that $h \circ f_0 (f^j_0(q)) = f_0^{j+1}(q)$. Hence,
\[
f^n(p) = \left(h \circ f_0\right)^{n}(p) = \left( h \circ f_0 \right)\circ f^{n-1}_0(p) = f_n \circ f^{-n}_0 \circ f_0^n ( p ) = f_n(p). 
\] 
This solves the case where $p$ is a periodic point.
\subsection*{Step 6} 
To finish the proof we have to show that we can always find a periodic point $p$ of period $n$ and stable index $2$ such that the eigenvalues of $Df^{n}(p)$ are all positive. To obtain a periodic point of stable index $2$, it is enough to remember that in the isotopy class of any partially hyperbolic diffeomorphism of a $3$-dimensional manifold, there exists a partially hyperbolic diffeomorphism with periodic points of different indices \cite{Ma1978}. After a small perturbation, we may assume that there is a transverse homoclinic point associated to this periodic point. Smale's theorem \cite{smale65} then implies the presence of a horseshoe, therefore of periodic orbits with any arbitrary code. In each ``block'' of the horseshoe, a combination of the orientations of the axes is either preserved or reversed. Choosing a suitable itinerary we can obtain periodic points $p$ of period $n$ such $Df^{n}(p)$ preserves the orientation of all axes, hence all eigenvalues of $Df^{n}(p)$ are positive.

\section{New examples and final remarks} \label{section.finalexamples}

In this section we will mention a few new classes of examples of stably ergodic diffeomorphism on $3$-dimensional manifolds. Before that, let us explain the examples that were already known which were considered by \cite{T2004}. 

Let $A\in \mathrm{SL}(3,\Z)$ be a hyperbolic matrix with eigenvalues $0< \lambda^{ss} < \lambda^{ws} <1< \lambda^{uu}$, and with eigendirections $E^{*}_A$, for $* = ss,ws, uu$ respectively, and let $f_A$ be the Anosov diffeomorphism induced by $A$ on $\T^3$. The construction then is done by fixing a small ball $B \subset \T^3$ around a fixed point $p\in \T^3$, and making a deformation supported in $B$ such that one obtains a diffeomorphism $f$ with the following properties: $f$ has a dominated splitting of the form $T\T^3 = E^{uu}_f \oplus E^{cs}_f $ such that the direction $E^{cs}_f$ is contained in a small cone around the direction $ E^{ws}_A \oplus E^{ss}_A$, and the direction $E^{uu}_f$ is contained in a small cone around the direction $E^{uu}_A$; $f$ coincides with $A$ outside $B$; the direction $E^{cs}_f$ does not admit any further dominated decomposition.

The diffeomorphism $f$ constructed above has the feature that the direction $E^{cs}_f(x)$ contracts uniformly for any point $x \in \T^3/B$. This is feature is crucial to prove, using a combinatorial argument, that if $B$ is sufficiently small then $f$ verifies the following: there exists $c>0$ such that Lebesgue almost every $x$ verifies
\begin{equation}\label{eq.mostlycontracting}
\displaystyle \limsup_{n\to +\infty} \frac{1}{n} \sum_{j=0}^{n-1} \log \|Df(f^j(x))|_{E^{cs}_f}\| < -c.
\end{equation}
Condition \eqref{eq.mostlycontracting} is sometimes called \emph{mostly contracting}, and it is crucial in Tahzibi's proof of stable ergodicity of $f$. The known examples of stably ergodic systems that are not partially hyperbolic have this property that the direction $E^{cs}_f$ contracts uniformly outside some fixed region in the manifold. Let us see some new examples for which Tahzibi's strategy does not apply.

\subsection*{Isotopic to Anosov$\times$Identity}

Let $g:\T^2 \to \T^2$ be a volume preserving Anosov diffeomorphism, and let $Id:S^1 \to S^1$ be the identity. Consider the volume preserving diffemorphism on $h = g \times Id$ and let $p$ be a fixed point for $h$. By the construction done in the proof of Theorem \ref{main.isotopy},  there exists a stably ergodic diffeomorphism $f$ isotopic to $h$, such that $f$ preserves a dominated splitting $T\T^3=  E^{uu} \oplus E^{cs}$ and such that $f|_{\T^3/B_{\varepsilon}(p)}$ is $C^1$-close to $h|_{\T^3/B_{\varepsilon}(p)}$, for some $\varepsilon>0$ small.

\subsection*{Isotopic to the time-one map of the geodesic flow}

Let $(S,g)$ be a compact riemannian manifold with a metric $g$ of constant negative curvature. Up to re-scaling $g$, we may suppose that $g$ admits a closed geodesic of length $1$. Let $T^1S$ be the unit tangent bundle of $S$ and $g_t:T^1S \to T^1S$ be the geodesic flow defined by $g$. Consider the diffeomorphism $h = g_1$ and observe that $h$ is partially hyperbolic with a dominated splitting $T T^1S = E^{uu} \oplus E^c \oplus E^{ss}$. Also observe that $h$ is isotopic to the identity. Since we assumed that there is a closed geodesic of length $1$, we have that there exists a point $p\in T^1S$ which is a fixed point for $h$. 

Arguing in the same way as before, there exists a stably ergodic diffeomorphism $f$ which is isotopic to $h$ (in particular, isotopic to the identity), and such that $f$ is not partially hyperbolic but admits a dominated splitting of the form $T T^1S = E^{uu} \oplus E^{cs}$.

\subsection*{Isotopic to anomalous partially hyperbolic diffeormorphisms}
New examples can be obtained in even stranger manifolds by suitably perturbing the examples obtained by Bonatti, Gogolev and Potrie in \cite{BGP2016}. These examples are in $3$-manifolds with fundamental groups of exponential growth, and they are not in the isotopy class of the identity. Bonatti, Gogolev, Hammerlindl and Potrie in \cite{BGHP2017} proved even more general examples of this kind that are not dynamically coherent. 

\section*{Acknowlegements:} We want to thank R. Ures for useful comments and remarks that helped improve this paper. We also thank S. Crovisier, and T. Fisher for useful conversations. 
\bibliographystyle{alpha}
\bibliography{Draft_Davi_Jana_Gabriel}

\end{document}